\documentclass[12pt, reqno]{amsart}
\usepackage{amsmath, amsthm, amscd, amsfonts, amssymb, graphicx, color}
\usepackage[bookmarksnumbered, colorlinks, plainpages]{hyperref}
\hypersetup{colorlinks=true,linkcolor=red, anchorcolor=green, citecolor=cyan, urlcolor=red, filecolor=magenta, pdftoolbar=true}

\usepackage{amsfonts}
\usepackage{color,graphicx,shortvrb}
\usepackage{enumerate}
\usepackage{amsmath}
\usepackage{amssymb}
\usepackage{graphicx}

\usepackage{enumerate}
\usepackage{amsmath}
\usepackage{amssymb}

\usepackage[latin 1]{inputenc}

\usepackage{amsmath}

\newtheorem{theorem}{Theorem}[section]

\newtheorem{corollary}[theorem]{Corollary}

\newtheorem{proposition}[theorem]{Proposition}

\newtheorem{remark}[theorem]{Remark}
\newtheorem{example}[theorem]{Example}

\usepackage{centernot}

\usepackage{tikz}
\usetikzlibrary{arrows,matrix}

\begin{document}

\title[$M(r,s)$-properties and the uniform Kadec-Klee property in JB$^*$-triples]{Geometric implications of the $M(r,s)$-properties and the uniform Kadec-Klee property in JB$^*$-triples}

\author[Li]{Lei Li}
\address[Li]{School of Mathematical Sciences and LPMC,
  Nankai University, Tianjin 300071,  China}
\email{leilee@nankai.edu.cn}

\author[Nieto]{Eduardo Nieto}

\author[Peralta]{Antonio M. Peralta}
\address[Nieto and Peralta]{Departamento de An\'{a}lisis Matem\'{a}tico, Facultad de Ciencias, Universidad de Granada, 18071 Granada, Spain}
\email{enieto@ugr.es, aperalta@ugr.es}
\date{\today}
\thanks{First author partially supported by NSF of China (11301285 and 11371201). Second and third authors partially supported by Junta de Andaluc\'{\i}a grant FQM375 and the Spanish Ministry of Economy and Competitiveness and European Regional Development Fund project no. MTM2014-58984-P}

\maketitle
\begin{abstract}
We explore new implications of the $M(r,s)$ and $M^*(r,s)$ properties for Banach spaces. We  show that a Banach space $X$ satisfying property $M(1,s)$ for some $0<s\leq 1$, admitting a point $x_{0}$ in its unit sphere at which the relative weak and norm topologies agree, satisfies the generalized Gossez-Lami Dozo property. We establish sufficient conditions, in terms of the $(r, s)$-Lipschitz weak$^*$ Kadec-Klee property on a Banach space $X$, to guarantee that its dual space satisfies the UKK$^*$ property. We determine appropriate conditions to assure that a Banach space $X$ satisfies the $(r, s)$-Lipschitz weak$^*$ Kadec-Klee property. These results are applied to prove that every spin factor satisfies the UKK property, and consequently, the KKP and the UKK properties are equivalent for real and complex JB$^*$-triples.
\end{abstract}

\section{Introduction}

Banach spaces which are $M$-ideals in their bidual have been intensively studied during decades due the rich geometric and isometric properties that they enjoy (compare the monograph \cite{HarWerWer}). J.C. Cabello and the second author of this note explore a weaker notion in \cite{CN1998}. Following the just quoted reference, a Banach space $X$ satisfies the \emph{$M(r,s)$-inequality} (with $r, s \in ]0, 1]$) if $$ \|x^{***}\| \geq r\|\pi (x^{***})\| + s\|x^{***} - \pi (x^{***})\|,
$$ for all $x^{***}$ in the third dual, $X^{***}$, of $X$, where $\pi$ is the canonical
projection of $X^{***}$ onto $X^{*}$. The $M(1, 1)$-inequality is the classical
notion of $M$-ideal \cite{HarWerWer}. Inspired by the so called properties $(M)$ and $(M^*)$, studied in \cite{HarWerWer}, J.C. Cabello and the second author of this note introduce in \cite{CN2000} properties $M(r, s)$ and $M^*(r, s)$. For the concrete definition, we recall that, given $r, s \in \left ]0, 1 \right ]$, a Banach space $X$ has {\it property} $M(r, s)$ (resp. $M^* (r, s)$) if whenever $u, v \in X$ (resp. $X^*$) with $\|u\| \leq \|v\|$ and
$(x_{\alpha})$ is a bounded weakly (resp. weak$^*$-) null net in $X$ (resp. $X^*$), then
$$
\limsup_{\alpha}\|ru + sx_{\alpha}\| \leq \limsup_{\alpha}\|v + x_{\alpha}\|.
$$  It should be noted here that properties $M(1,1)$ and $M^* (1, 1)$ are precisely properties $(M)$ and $(M^*)$ in the terminology of \cite{HarWerWer}.\smallskip

Properties $M(r,s)$ have been successfully applied in fixed point theory (compare \cite{GarFalSim, DomGarFalJap1998, CN2000}). %For example, J. Garc{\'i}a-Falset and B. Sims showed in \cite{GarFalSim}  that for a Banach space $X$ satisfying property $(M)$, nonexpansive self mappings of nonempty weakly compact convex sets in $X$ necessarily have fixed points (i.e. $X$ has the weak fixed point property). T. Dom{\'i}nguez-Benavides, J. Garc{\'i}a-Falset, and M.A. Jap{\'o}n prove in \cite{} that $X$ satisfies the weak$^*$ fixed property whenever it has property $(M^*).$ Actually, a Banach space satisfying property $M(r,s)$, with $r+s>1$, always has the weak fixed property, and if $X$ admits a shrinking basis strongly bimonotone and property $M(r,s)$ with $r+s>1$, then $X^*$ satisfies the weak$^*$ fixed point property (compare \cite{CN2000}).
\smallskip

The fixed point theory motivated the developing of many interesting properties in Banach space theory. That is also the case of the \emph{Kadec-Klee (or Radon-Riesz) property} (KKP in the sequel). We recall that a Banach space has the KKP if any sequence in the unit sphere whose weak limit is also in the unit sphere, is indeed norm convergent. The uniform Kadec-Klee property for the weak topology on a Banach space was introduced by R. Huff \cite{Hu} as a useful substitute for uniform convexity, especially in many non-reflexive spaces. D. van Dulst and B. Sims \cite{DS}%, building on work of M.S. Brodskii and P.D. Milman \cite{BM}, W.A. Kirk \cite{Ki1} and F.E. Browder \cite{Br}, 
 showed that the uniform Kadec-Klee property for weak and weak$^*$ topologies implies weak (resp. weak$^*$) normal structure. %This geometric property implied in turn, via \cite{Ki1}, that every norm non-expansive mapping on a weakly (resp. weak$^*$) compact, convex set must have a fixed point. (The weak-star result is contained in \cite{DS} and is also implicit in \cite{Ki2}.) C. Lennard \cite{Le2} extended these results to a wider class of weaker, possibly non-locally convex topologies on a Banach space; including the topology of local convergence in measure on the function space $L^1[0, 1]$. %(Besbes \cite{Be1} obtained the same normal structure theorem at about the same time).For a more detailed discussion of the relationships between fixed point properties, normal structures and uniform Kadec-Klee type properties, we refer the interested reader to \cite{Le2}, and the reference contained therein.
\smallskip

It is well-known (see \cite{Ar, Si}) that the Schatten $p$-classes $C_p$, $1\leq p<\infty$, have the KKP. It has been shown by C. Lennard \cite{Le1} that the direct argument given by J. Arazy \cite{Ar} for trace-class operators can be refined to show that $L^{1} (H)$, the space
of trace class operators on an arbitrary infinite-dimensional Hilbert space $H$ satisfies a stronger property called the uniformly Kadec-Klee in the weak$^*$ topology UKK$^*$ (see section \ref{sec: def} for the detailed definitions). From a somewhat different viewpoint, it is a classical theorem of F. Riesz that norm convergence for sequences in the unit sphere of $L^1[0,1]$ coincides with convergence in measure. Appropriate uniform version of this theorem may be found in \cite{Le2}.%, and extensions to certain Lorentz spaces in \cite{Le3}.
\smallskip

Independently from the fixed point theory, The Kadec-Klee property has been deeply studied in certain particular classes of Banach spaces including C$^*$-algebras and JB$^*$-triples in connection with the Alternative Dunford-Pettis property (compare \cite{AP, BP04, BP05} and \cite{BunPe07}). Proposition 2.13 in \cite{BP04} provides a complete description of those JB$^*$-triples satisfying the KKP, namely, a JB$^*$-triple satisfies this property if and only if it is finite-dimensional or a Hilbert space or a spin factor. A similar conclusion holds for real JB$^*$-triples (compare \cite[Proposition 3.13]{BP05}). It is a natural open problem to ask wether a JB$^*$-triple satisfying the KKP satisfies or not the stronger UKK property. This is one of the motivations for this note.\smallskip

In section \ref{sec: Mrs prop} we explore new geometric implications of the $M(r,s)$ and $M^*(r,s)$ properties. In \cite[Theorem 2.4 and Corollary 2.5]{GarFalSim},
J. Garc\'{\i}a-Falset and B. Sims proved that if a Banach space $X$ has
property $(M)$, and $S_{X}$ contains a point at which the relative weak and norm topologies agree, then $X$ has $w$-ns. We establish a generalization of this result by showing that a Banach space $X$ satisfying property $M(1,s)$ for some $0<s\leq 1,$ and admitting a point $x_{0}$ in its unit sphere at which the relative weak and norm topologies agree, satisfies the generalized Gossez-Lami Dozo property (see Proposition \ref{p generalized GarciaFalset and Sims}). In Theorem \ref{t 22} we prove that for a Banach space $X$ satisfying property $M^* (1,s)$ for some $0<s\leq 1$, reflexivity of $X$ admits many equivalent reformulations in terms of classical properties like Radon-Nikod{\'y}m property, PCP, and KKP.\smallskip

Proposition \ref{Proposition 2.4} provides sufficient conditions, in terms of the $(r, s)$-Lipschitz weak$^*$ Kadec-Klee property on a Banach space $X$, to guarantee that its dual space satisfies the UKK$^*$ property (see definitions below). Theorem \ref{proposition 3.1.} sets appropriate conditions to assure that a Banach space $X$ satisfies the $(r, s)$-Lipschitz weak$^*$ Kadec-Klee property. These two results, appropriately combined, allow us to establish, in section \ref{sec: UKKP spin}, that every spin factor satisfies the UKK property (Theorem \ref{spin01}), and consequently, the KKP and the UKK properties are equivalent for real and complex JB$^*$-triples (see Corollary \ref{c UKK equals KKP} and Proposition \ref{p UKK equals KKP for real}). We also obtain, with similar arguments and tools, another proof of the result, established by C. Lennard, which asserts that $L^{1} (H)$ satisfies the UKK$^*$ property (Theorem \ref{t Lennard trace class}).

\section{Background and basic definitions}\label{sec: def}

Throughout the paper, the closed unit ball and the unit sphere of a given Banach space $X$ will be denoted by $B_{X}$ and $S_{X}$, respectively.\smallskip

A sequence $(x_n)$ in a Banach space $X$ is called \textit{separated} (respectively, \textit{$\varepsilon$-separated}) if sep$(x_n):=\inf\{\|x_n-x_m\|: n\neq m\}>0$ (respectively, $\geq\varepsilon$). It is known that $X$ has the \textit{Kadec-Klee property} if and only if every separated weakly convergent sequence $(x_n)$ in the closed unit ball of $X$ converges to an element of norm strictly less than one.\smallskip

Following standard notation, we say that $X$ satisfies the \textit{uniform Kadec-Klee} (UKK) \emph{property}\label{def UKKP} if for each $\varepsilon>0$ there exists $\delta>0$ such that for
every sequence $(x_n)$ in the closed unit ball of $X$ satisfying sep$(x_n)\geq \varepsilon$ and $(x_n)\to x$ weakly, then it holds that $\|x\|<1-\delta$ (see \cite{Hu}).\smallskip

Clearly, if $X$ has the Schur property (i.e. norm and weak convergent sequences in $X$ coincide) or if $X$ is uniformly convex then $X$ has the UKK property.
While uniformly convex spaces are necessarily reflexive, it turns out that many classical non-reflexive spaces, e.g. the Hardy space $H^1$ of analytic functions on the ball or on the polydisc in $\mathbb{C}^{N}$ \cite{BDDL94}, the Lorentz space $L_{p,1}(\mu)$ \cite{CDLT91, DH94} and the trace class operators $L^1(H)$ \cite{Hu}, satisfy the UKK or the UKK$^*$ property.\smallskip

A Banach space $X$ is said to be \textit{nearly uniformly convex} if for every $\varepsilon>0$, there exists a $\delta$, $0<\delta<1$ such that for any sequence $\{x_n\}$ in the closed unit ball $B_X$ of $X$ with $\text{sep}(x_n) \geq \varepsilon,$
then $\text{conv}(\{x_n\})\cap (1-\delta)B_X \neq \emptyset$.\smallskip

We recall next the definition of the uniform weak$^*$ Kadec-Klee property, which is somehow due to B. Sims \cite{S86}. Suppose that $X$ is a dual Banach space and $\varepsilon>0$. $X$ is \textit{$\varepsilon$-uniformly Kadec-Klee in the weak*-topology} ($\varepsilon$-UKK$^*$ in the sequel) if there exists $\delta\in (0,1)$ such that whenever $C$ is a weak$^*$ compact, convex subset of unit ball $B_X$ with $\sup\{\text{sep}(x_n): (x_n)\subset C\}\geq \varepsilon$ it follows that $C\cap (1-\delta)B_X \neq \emptyset$.
$X$ is said to have the \textit{uniform Kadec-Klee property in the weak*-topology} (UKK$^*$ property) if it is $\varepsilon$-UKK$^*$ for all $\varepsilon>0$.\smallskip

If $B_{X^*}$ is weak$^*$ sequentially compact, the weak$^*$ uniform Kadec-Klee property can be reformulated in terms of the definition given for UKK property in page \pageref{def UKKP}, but replacing weak convergence with weak$^*$ convergence (see \cite[\S 3]{DS}).

\section{Banach spaces satisfying the $M(r,s)$-properties}\label{sec: Mrs prop}

In this section we shall revisit the additional properties satisfied by Banach spaces possessing $M(r,s)$-properties in connection with previous contributions.\smallskip

We recall that a Banach space $X$ has {\it weak} {\it
normal structure} (in short {\it w-ns}) if every weakly compact convex subset $K$ of $X$ containing more than one point admits a {\it nondiametral point}, that is, there exists $x$ in $K$ such that $$\displaystyle{{\rm diam}\ (K)> \sup\left\{\|x - y\|:\ y \in K\right\}}.$$ When $X$ is a dual Banach space, if in the above definition, the weak topology is replaced with the weak$^*$-topology we say that $X$ has \emph{weak$^*$ normal structure} (w$^*$-ns).\smallskip

Another related property is the \emph{generalized Gossez-Lami Dozo property} (GGLD property in short). A Banach space $X$ satisfies the GGLD property if for every weakly null sequence $(x_{n})$ in $X$ such that $\lim_{n} \|x_{n}\| = 1$, we have that
$\displaystyle{D[(x_{n})] > 1}$, where
$$
D[(x_{n})]  = \limsup_{n}\left (\limsup_{m}\|x_{n} - x_{m}\|\right
).
$$ It is known that every Banach space satisfying the GGLD property has $w$-ns (see \cite{JiMel1992}).\smallskip

In \cite[Theorem 2.4 and Corollary 2.5]{GarFalSim},
J. Garc\'{\i}a-Falset and B. Sims proved that if a Banach space $X$ has
property $(M)$, and $S_{X}$ contains a point at which the relative
weak and norm topologies agree, then $X$ has $w$-ns. Making use of a different technique, we show next that the class of Banach spaces for which the conclusion of the above result of Garc\'{\i}a-Falset and Sims holds also includes Banach spaces satisfying property $M(1,s)$ for some $0<s\leq 1$. We prove that a stronger conclusion is also true. In Example \ref{example CaNiOj 4.4} below we present a Banach space satisfying the $M(1,s)$ property for a fixed $0<s<1$, but failing the $(M)$ property.

\begin{proposition}\label{p generalized GarciaFalset and Sims}
Let $X$ be a Banach space having property $M(1,s)$ for some $0<s\leq 1$. If there exists a point $x_{0} \in
S_{X}$ at which the relative weak and norm topologies agree, then $X$ has the GGLD property.
\end{proposition}

\begin{proof}
Suppose that $X$ fails to have the {\it
GGLD} property. Then there exists a weak null sequence $(x_{n})$
in $X$ satisfying
$$
\lim_{n}\|x_{n}\| = 1 \phantom{ab} {\rm and}\phantom{ab} D
[(x_{n})]\leq  1.
$$

Let $x \in X$ with $\|x\| < 1$. Then, for $m$ large enough we have $\|x\|
\leq \|x_{m}\|$. So, we deduce, by property $M(1,s)$, that
$$
\limsup_{n}\|x - sx_{n}\|\leq \limsup_{n}\|x_{m} - x_{n}\|.
$$

Hence, taking limit on $m$, we have that $\limsup_{n}\|x -
sx_{n}\| \leq 1$. Therefore, by the triangular inequality, $\limsup_{n}\|x -
sx_{n}\| \leq 1$ for all $x \in B_{X}$. In particular, by the weak
lower semi-continuity of the norm, $\limsup_{n}\|x_{0} - sx_{n}\|
= \|x_{0}\|=1$. So, we deduce from the hypothesis on $x_0$ that $\lim_{n}\|x_{\sigma(n)}\| = 0$, for a suitable subsequence $(x_{\sigma(n)})$ of $(x_n)$,  which is a contradiction.
\end{proof}

We recall that a Banach space $X$ has the \emph{point of continuity property} (\emph{PCP} in short) if every non-empty (weakly) closed subset $K$ of $B_X$ admits a point at which the identity map on $K$ is weak-norm continuous (compare \cite{EdgarWheeler84}, \cite[\S III]{GhossGodMauSchacher}).\smallskip

A wide list of geometric properties are equivalent for Banach spaces satisfying $M^*(1,s)$ property for some $1\geq s>0$ (compare \cite[Corollary 2.5]{GarFalSim}).\smallskip

\begin{theorem}\label{t 22} Let $X$ be a Banach space having property $M^* (1,s)$ for some $0<s\leq 1$. Then the following assertions are
equivalent:
\begin{enumerate}[$(i)$]
\item $X$ is reflexive;
\item $X$ has the Radon-Nikod{\'y}m property;
\item $X$ has the PCP;
\item There is a point in $S_{X}$ at which the relative weak and norm topologies sequentially agree;
\item $X$ has the KKP.
\item $X$ doesn't contain an isomorphic copy of $c_{0}$.
\end{enumerate}
\end{theorem}

\begin{proof} Let us start with an observation. By \cite[Proposition 3.1]{CN2000}, $X$ satisfies property $M(1,s)$ and the $M(1,s)$-inequality.

It is well known that every reflexive space has the Radon-Nikod{\'y}m property so $(i) \Rightarrow (ii)$.  The implications $(ii)\Rightarrow (iii) \Rightarrow (iv),$ $(v)\Rightarrow (iv)$, and $(i) \Rightarrow (vi)$ are also well known.\smallskip

$(iv) \Rightarrow (v)$ Let $x_{0}$ be a point in $S_{X}$ at which the relative weak and norm
topologies on $S_{X}$ sequentially agree. Let $x \in S_{X}$ and let
$(x_{n})$ be a sequence in $X$ with $(x_{n})
\stackrel{w}{\longrightarrow} x$ and $\|x_{\alpha}\|
\longrightarrow 1$. By property $M(1,s)$ and the weak lower
semi-continuity of the norm,
$$
1 \leq \limsup_{n}\|x_{0} + s(x_{n} - x)\| \leq
\limsup_{n}\|x + (x_{n} - x)\| = 1.
$$ Thus, by assumptions, $\lim_{n}\|x_{n} - x\| = 0$.\smallskip

If $X$ is non-reflexive then, by \cite[Corollary 3.4$(1)$]{CN1998}, $X$ contains an isomorphic copy of $c_{0}$. Thus, we have $(vi) \Rightarrow (i)$. If we assume $(iv)$, Proposition \ref{p generalized GarciaFalset and Sims} implies that $X$ has the {\it GGLD} property. However, by \cite[Theorem 8]{Dow}, $X$ fails to have the {\it GGLD} property, which is a contradiction. This proves  $(iv) \Rightarrow (i)$.
\end{proof}

\begin{remark}\label{r M(1,s) and sequential weak norm continuity at a point} The arguments given in the proof of $(iv) \Rightarrow (v)$ in Theorem \ref{t 22} above, actually show that a Banach space $X$ satisfying the $M(1,s)$-property and admitting a point in $S_{X}$ at which the relative weak and norm topologies sequentially agree always satisfies the KKP.
\end{remark}

Following \cite[Definition 2.3]{GoKalLan2000}, given $r, s \in ]0, 1 ]$, we will say that a Banach space $X$ satisfies the {\it $(r, s)$-Lipschitz
weak$^*$ Kadec-Klee property} (in short, {\it $(r, s)$-LKK$^{*}$
property}) if for every $x^{*} \in X^{*}$ and every weak$^{*}$
null sequence $(x_{n}^{*} )$ in $X^{*}$,
$$
\limsup_{n}\| x^{*} + x_{n}^{*}\| \geq r\|x^{*}\| +
s\limsup_{n}\|x_{n}^{*}\| .
$$

It is obvious that the $(r, s)$-LKK$^{*}$ property implies
property $M^{*} (r,s)$ (property $M^{*} (r, s)$ and its
sequential version are equivalent for separable spaces -see
\cite[pp. 300, 301]{HarWerWer}).

\begin{corollary}\label{Corollary 2.3} Let $X$ be a Banach space satisfying the $(1, s)$-LKK$^*$ property. Then the relative weak and norm topologies on $S_{X}$ do not coincide at any point in $S_{X}$.
\end{corollary}

\begin{proof} Arguing by contradiction, we suppose the existence of a point $x_{0} \in
S_{X}$ at which the relative weak and norm topologies on $S_{X}$ coincide.\smallskip

By \cite[Proposition 2.5]{CN1998}, $X$ is an Asplund space. Thus, $X$ does not satisfy Schur property.
So, there exists a weakly null sequence $(x_{n})$ in $S_{X}$. We deduce, by the weak lower semi-continuity of the norm and \cite[Lemma 2.5]{GoKalLan2000}, that
$$
\limsup_{n}||x_{0} + s x _{n}|| = 1,
$$ which implies, from our assumptions, that $\|x_{\sigma(n)}\| \to 0$, for a certain subsequence $(x_{\sigma(n)})$, which is impossible.
\end{proof}

Let us fix some notation. Given $\varepsilon
> 0$ we denote
$$
\pi_{1,\varepsilon}^+ = \left\{(r, s) \in (0,1]^2: \frac{1-r}{s} < \varepsilon\right\}.% \hbox{  and }
%\pi_{\varepsilon,1}^+ = \left\{(r, s) \in (0,1]^2 : \frac{1-s}{r} < \varepsilon\right\}.
$$ For a Banach space $X,$ we write
$$
\hbox{LKK}^*(X) = \{(r, s) \in (0,1]^2 :\ X \phantom{a}\mbox{\rm satisfies the}\phantom{a}
(r, s)\mbox{\it -LKK}^* \phantom{a} {\rm property}\}.
$$

We establish now some results related to the UKK$^*$ property. We recall that for every separable Banach space $X$, the closed unit ball of $X^*$ is weak$^*$ sequentially compact. It is also known that $B_{X^*}$ is weak$^*$ sequentially compact whenever $X^*$ does not contain a copy of $\ell_1$. In
particular, $B_{X}$ is weak$^*$ sequentially compact whenever $X$ is reflexive (cf. \cite[Chapter XIII]{Die}).

\begin{proposition}\label{Proposition 2.4} Let $X$ be a Banach space. If $\hbox{LKK}^*(X) \cap \pi_{1,\varepsilon}^+ \neq \emptyset$ for all $\varepsilon > 0$ {\rm(}in
particular, if $X$ satisfies the $(1, s)$-LKK$^*$ property{\rm)}, then
$X^*$ has the UKK$^*$ property.
%
%\begin{enumerate}[$(i)$]
%\item If $X$ satisfies the $(r, s)$-LKK$^*$ property with $r
%+ s > 1$ and $X^*$ does not satisfy the Schur property, then $X^*$ has weak uniform
%normal structure.
%
%\item If $\hbox{LKK}^*(X) \cap \pi_{\varepsilon,1}^+ \neq \emptyset$ for all $\varepsilon
%> 0$ {\rm(}in particular, if $X$ satisfies the $(r, 1)$-LKK$^*$ property{\rm)}
%and $X^*$ is not a Schur space, then $X^*$ satisfies the uniform Opial property.
%
%\item If $\hbox{LKK}^*(X) \cap \pi_{1,\varepsilon}^+ \neq \emptyset$ for all $\varepsilon > 0$ {\rm(}in
%particular, if $X$ satisfies the $(1, s)$-LKK$^*$ property{\rm)}, then
%$X^*$ has the UKK$^*$ property.
%\end{enumerate}
\end{proposition}

\begin{proof} We observe that the hypothesis $\hbox{LKK}^*(X) \cap \pi_{1,\varepsilon}^+ \neq \emptyset$ for all $\varepsilon > 0$ implies that $B_{X^*}$ is weak$^*$ sequentially compact. Indeed, since the $(r, s)$-LKK$^{*}$ property implies property $M^{*} (r,s)$, we deduce that, for every $\varepsilon>0$, there exist $r,s\in (0,1]$ with $\frac{1-r}{s} <\varepsilon$ such that $X$ satisfies property $M^{*} (r,s)$. In particular, we can assure that $X$ satisfies the $M^{*} (r,s)$ property for certain $r,s$ with $r+s>1$. Proposition 3.1 \cite{CN2000} and Corollary 2.8 in \cite{CN1998} assure that $X$ does not contain an isomorphic copy of $\ell_1$, and hence $B_{X^*}$ is weak$^*$ sequentially compact.\smallskip

Let us fix an arbitrary $\varepsilon> 0$. Let $(x_{n}^*)$ be a sequence in
$B_{X^*}$ converging to $x^*$ in the weak$^*$ topology with ${\rm
sep}(x_{n}^*) \geq \varepsilon$ and consider $(r, s) \in \hbox{LKK}^*(X) \cap \pi_{1,\frac\varepsilon2}^+$.
Then we have
$$
1 \geq \limsup_{n}\|x_{n}^*\| = \limsup_{n}\|x^* + (x_{n}^*
- x^*)\| \geq r\|x^*\| + s\limsup_{n}\|x_{n}^* - x^*\| .
$$

On the other hand, since ${\rm sep}(x_{n}^*) \geq \varepsilon$, we have
that $\limsup_{n}\|x_{n}^* - x^*\| \geq \varepsilon/2$. Therefore,
$$ \|x^*\| \leq \frac{1 - s \frac\varepsilon2}{r} . $$ Since $\displaystyle{\frac{1-r}{s} < \frac{\varepsilon}{2}}$, we have that
$\displaystyle{\delta = 1 - \frac{1 - s \frac\varepsilon2}{r} > 0}$.
\end{proof}

We can establish now a key tool to deal with the $(r, s)$-LKK$^*$ property. We recall that $\limsup$ and $\liminf$ of a bounded net of real numbers can be defined in a similar manner as for sequences.% (see, for example, \cite[Definition 2.5.20 and Exercise 2.55]{Megg}).

\begin{theorem}\label{proposition 3.1.} Let $X$ be a Banach space
and $r, s \in ]0, 1]$. If there exists a net $(K_{j})$ of compact
operators on $X$ %with $\|K_{j}\| \leq 1$
satisfying %$\displaystyle\lim_{j} \| K_{j} (x) - x\|=0$ for all $x \in X$,
$\displaystyle \lim_{j} \|K_{j}^* (x^*)- x^*\|=0$ for all $x^* \in X^*$, and
\begin{equation}\label{eq shrinking in the space} \limsup_{j} \sup_{x, y \in B_{X}} \|r K_{j}(x) + s (y - K_{j} (y))\| \leq 1,
\end{equation} then $X$ satisfies the $(r, s)$-LKK$^*$ property.
\end{theorem}

\begin{proof} Let us first observe that, by \eqref{eq shrinking in the space}, we can easily deduce that $(K_{j})$ is a bounded net. Since for each $x^*\in B_{X^*}$ we have $$ r\|K_{j}^*
x^*\|+ s\|x^* - K_{j}^* x^*\| =  \sup_{x\in B_{X}} |r x^* K_j (x)| + \sup_{y\in B_{X}} |s x^*(y -K_j (y))|$$
$$= \sup_{x,y\in B_{X}} |r x^* K_j (x) +  s x^*(y -K_j (y))|\leq \sup_{x, y \in B_{X}} \|r K_{j}(x) + s (y - K_{j} (y))\|,$$ it can be easily seen that \begin{equation}\label{eq first ineq}
\limsup_{j}\sup_{x^* \in B_{X^*}}\left (r\|K_{j}^*
(x^*)\|+ s\|x^* - K_{j}^* (x^*)\|\right )\leq 1.
\end{equation} The equivalence of \eqref{eq shrinking in the space} and \eqref{eq first ineq} was already established in \cite[Proposition 2.3]{Ni2003} for non-necessarily compact operators. The argument above is included here for completeness reasons.\smallskip

Now, let us fix $x^* \in X^*$ and a weak$^*$ null sequence  $(x_{n}^*)$ in $X^*$. Given an arbitrary $\varepsilon>0$, there exists a natural $m_0$ such that $$\sup_{m\geq m_0} \|x^*+ x_m^* \| \leq \limsup_{m} \|x^*+x_m^* \| +\varepsilon.$$ Applying \eqref{eq first ineq}, we have
$$ \limsup_{j} \sup_{m\geq m_0} r \|K_{j}^* (x^* +
x_{m}^*)\|+ s \|(x^* + x_{m}^* ) - K_{j}^*(x^* +
x_{m}^*) \|
$$
$$
\leq \limsup_{m}\|x^* + x_{m}^*\|+\varepsilon.
$$ To simplify the notation, let us set $$a_j:=\sup_{m\geq m_0} r \|K_{j}^* (x^* +
x_{m}^*)\|+ s \|(x^* + x_{m}^* ) - K_{j}^*(x^* +
x_{m}^*) \|.$$

Since $\limsup_j a_j \leq \limsup_{m}\|x^* + x_{m}^*\|+\varepsilon$,  there exists $j_0$ such that for each $j\geq j_0$ we have $\sup_{l\geq j} a_l < \limsup_j a_j +\varepsilon\leq \limsup_{m}\|x^* + x_{m}^*\|+2 \varepsilon$. That is, the inequality $$ r \|K_{j}^* (x^* +
x_{m}^*)\|+ s \|(x^* + x_{m}^* ) - K_{j}^*(x^* +
x_{m}^*) \| < \limsup_{m}\|x^* + x_{m}^*\|+2 \varepsilon,$$ holds for every $j\geq j_0$ and $m\geq m_0$.\smallskip

Since for a fixed subindex $j$, $\lim_m \|K_j^* (x_m^*) \|=0,$ fixing an arbitrary $j\geq j_0$ and taking limit $\limsup_m$ in the above inequality, we get $$ r \|K_{j}^* (x^* )\|+ s \limsup_m \|(x^* + x_{m}^* ) - K_{j}^*(x^* ) \| \leq \limsup_{m}\|x^* + x_{m}^*\|+2 \varepsilon,$$ for every $j\geq j_0.$ Finally, taking $\limsup_{j\geq j_0}$ we deduce from the hypothesis that  $$r \|x^*\|+ s \limsup_m \| x_{m}^* \| \leq \limsup_{m}\|x^* + x_{m}^*\|+2 \varepsilon,$$ the desired conclusion follows from the arbitrariness of $\varepsilon$.
\end{proof}

\begin{remark}\label{remark shrinking compact approximation inequality in the dual space}{\rm We observe that, accordingly to the arguments given in the proof of Theorem \ref{proposition 3.1.}, the assumption in \eqref{eq shrinking in the space} can be replaced with the inequality \eqref{eq first ineq} and the conclusion of the Theorem remains unaltered (c.f. \cite[Proposition 2.3]{Ni2003}).}
\end{remark}

Let us observe that a bounded net $(K_j)$ of compact operators on a Banach space $X$ satisfying
$\displaystyle\lim_{j} \| K_{j} (x) - x\|=0$ for all $x \in X$ and $\displaystyle \lim_{j} \|K_{j}^* (x^*)
- x^*\|=0$ for all $x^* \in X^*$, is termed a \emph{shrinking compact approximation of the identity} in many references (see for example \cite[Definition VI.4.16]{HarWerWer}).\smallskip

We complete this section with a series of example that illustrate the optimality and novelty of our previous results.

\begin{example}\label{example fails KKP and M1s and satisfies LKKstar}{\rm For $1 < p < +\infty$, consider
the equivalent renorming of $\ell_{p}$, $X = \mathbb{C}
\oplus_{\infty} \ell_{p}$, where the usual norm is considered on
$\ell_{p}$. Then $X$ satisfies the $(r, s)$-LKK$^*$ property for
all positive $r$ and $s$ with $r^{p} + s^{p} \leq 1$. Indeed, Let $\pi_{n},$ ($n \in \mathbb{N}$), denote the
natural projection of $X$ onto the first $n$th coordinates.
Given $(\alpha, x), (\beta, y) \in B_{X}$ and $n \in \mathbb{N}$, we
have
$$
\|r \pi_{n+1} (\alpha, x) + s ((\beta, y) - \pi_{n+1}(\beta, y))\|^{p} =
$$
$$
= \max\left\{(r |\alpha |)^{p}, r^{p} \sum_{i=1}^{n} |x_{i}|^{p} +
s^{p} \sum_{j=n+1}^{+\infty} |y_{j}|^{p}\right\} \leq r^{p} + s^{p} \leq 1.
$$ Theorem \ref{proposition 3.1.} implies that $X$ satisfies the $(r, s)${\it
-LKK}$^*$ property for evert $r,s\in (0,1]$ with $r^{p} + s^{p} \leq 1$.\smallskip

We claim that $X$ fails to have the KKP and property $M(1,s)$ for every $0<s\leq 1$. To see this, let $\{e_{n}\}$ denote the canonical basis in $\ell_{p}$, and we define $x_{n} = (1,e_{n})$,
its is clear that $\|x_{n}\| = 1$ and $(x_{n})
\stackrel{w}{\longrightarrow} x_0:=(1, 0)$. However $\|x_{n} - x_0\| = \|e_{n}\| = 1$, for every natural $n$. Clearly, $X$ is reflexive. If $X$ had property $M(1,s)$ for a real $s\in (0,1]$, Theorem \ref{t 22} would imply that $X$ has KKP, which is impossible. %Finally, since $X$ is reflexive, there exists a point $x_0$ in $S_{X}$ at which the relative weak and norm topologies agree. If $X$ had property $M(1,s)$, Remark \ref{r M(1,s) and sequential weak norm continuity at a point} would imply that $X$ satisfies KKP, which is a contradiction.
}\end{example}

We present next a non-reflexive example.
It is well known \cite[Theorem 3.b.9 and Theorem 3.d.4]{fegam} that the James tree space $JT$ has $w$-ns and the KKP. Applying Proposition \ref{Proposition 2.4}, we improve these facts, obtaining
stronger conclusion.

\begin{example}\label{example JT predual}{\rm Let $\mathcal{B}$ denote the predual of the James tree space $JT$ (compare \cite[page 175]{fegam}). Then the following statements hold:
\begin{enumerate}[$(a)$]
\item $\mathcal{B}$ satisfies the $(r, s)$-LKK$^*$ property for all $r, s
> 0$ with $r^2 + s^2 \leq 1$;
\item JT has the UKK$^*$ property;
\item $\mathcal{B}$ doesn't satisfy property $M^*(1,s)$ for any $0<s\leq 1$.
\end{enumerate}

To see the first statement, we recall that $JT$ is separable and admits a (boundedly complete) basis
$(e_{n})$ (compare \cite[Definition 3.a.2]{fegam}). We denote by $\pi_{n}$, $n \in \mathbb{N}$, the natural
projections onto the $n$th first coordinates of this basis. By \cite[Lemma
3.a.3]{fegam} the inequality
$$
\|\pi_{n} (x)\|^2 + \|x-\pi_{n}(x)\|^{2}\leq \|x\|^{2},
$$ holds for every $n \in \mathbb{N}$ and $x \in JT$.
So, for every $n \in \mathbb{N}$ and $x^* , y^* \in B_{JT^*}$,
$$
\|r \pi_{n}^* (x^*) + s (y^* - \pi_{n}^* (y^*) )\|^2 \leq
r^2 \|x^* \|^2 + s^2 \|y^*\|^2 \leq r^2 + s^2 \leq 1.
$$ Theorem \ref{proposition 3.1.} assures that $B$ satisfies the $(r, s)${\it
-LKK}$^*$ property.\smallskip

Having in mind that $JT$ does not contain $\ell_1$ (see \cite[Theorem 3.a.8]{fegam}), and hence $B_{JT}$ is weak$^*$ sequentially compact, statement $(b)$ follows from Proposition \ref{Proposition 2.4}.\smallskip

$(c)$ It is well known that $\mathcal{B}$ has the PCP  (so, $\mathcal{B}$ doesn't contain an
isomorphic copy of $c_{0}$) and fails to have the Radon-Nikod{\'y}m property (see \cite[4.b.2,  4.b.5, and 3.c.10]{fegam}). Since $\mathcal{B}$ is non-reflexive and $JT$ does not contain a copy of $\ell_1$, an application of \cite[Proposition 3.1]{CN2000} and \cite[Corollary 3.4$(2)$]{CN1998} gives the desired conclusion.}
\end{example}

We consider next a series of examples inspired in \cite{CNO} and \cite{CN2000}.

\begin{example}\label{example CaNiOj 4.4}{\rm(Compare \cite[Example 4.4]{CNO} and \cite[Proposition 4.1 and Corollary 4.2$(2.)$]{CN2000}) Fix a series $\displaystyle \sum_{n}\alpha_n$ of positive real numbers such that $\alpha_1\neq \alpha_2$ and $\displaystyle  \sum_{n=1}^{\infty} \alpha_n = a \in (0,1)$. Let $X$ be the Banach space $c_0$ equipped with the norm given by $$\|x\| := \sup\left\{|x_n| + \sum_{j=1}^{n} |x_j| \alpha_j : n\in \mathbb{N}\right\}.$$ By \cite[Proposition 4.1 and Corollary 4.2$(2.)$]{CN2000} we know that $X$ satisfies the $M^* (1, 1-a)$ property, and hence the $M (1, 1-a)$ property (see \cite[Proposition 3.1]{CN2000}). The arguments in the latest reference also prove that $X$ satisfies the hypothesis of Theorem \ref{proposition 3.1.}, and hence $X$ satisfies the $(1, 1-a)$-LKK$^*$ property.\smallskip

We claim that $X$ does not satisfy the $(M)$-property. Indeed, let $x= \frac{1}{1+\alpha_1} e_1$ and $y = \frac{1}{1+\alpha_2} e_2$, where $(e_n)$ is the canonical basis of $c_0$. It is not hard to check that $\|x\|=\|y\|=1$, $\|x+e_n\| = 1+\frac{\alpha_1}{1+\alpha_1} + \alpha_n \to 1+\frac{\alpha_1}{1+\alpha_1}$ and $\| y + e_n\| = 1 + \frac{\alpha_2}{1+\alpha_2} + \alpha_n \to 1 + \frac{\alpha_2}{1+\alpha_2} $, which shows that $X$ does not satisfy property $(M)$.}
\end{example}

\begin{example}\label{example CaNiOj 4.5}{\rm(Compare \cite[Proposition 4.1 and Corollary 4.3]{CN2000}) Let $X$ be the Banach space $\mathbb{C}\times c_0$ equipped with the norm defined by $$\|(\alpha,x)\| := \max\left\{|\alpha|+ \lambda \|x\|_{0}, \|x\|_0\right\},$$ where $0<\lambda<1$ is a fixed number.
Corollary 4.3 in \cite{CN2000} and its proof show that $X$ satisfies the $M^* (1-\lambda, 1)$ property, and hence the $M (1-\lambda, 1)$ property (see \cite[Proposition 3.1]{CN2000}).\smallskip

We shall finally prove that $X$ does not satisfy the $M(1,s)$-property for any $0<s\leq 1$. Indeed, let $x=(1, 0)$ and $y = (0,e_1)$, where $(e_n)$ is the canonical basis of $c_0$. Clearly, $\|x\|=\|y\|=1$. We can easily see that $\|x+ s (0,e_n)\| = 1+ \lambda s$ and $\| y + (0,e_n)\| = 1 $, for every $n\geq 2$,  which proves the desired statement.}
\end{example}

\section{Uniform Kadec-Klee property in JB$^*$-triples}\label{sec: UKKP spin}

A \textit{spin factor} is a JB$^*$-triple $V$ whose norm and triple product are given by the following rules. $V$ is a Hilbert space with respect to an Hilbert product $\langle\cdot,\cdot\rangle$, there exists a \textit{conjugation} $\overline{\ \cdot\ }: V\to V$ (i.e. a conjugate linear isometry of period 2) such that
\begin{eqnarray}\label{22}
\{a,b,c\}=\frac{1}{2}(\langle a,b\rangle c+\langle c,b\rangle a-\langle a,\overline{c}\rangle \overline{b}),
\end{eqnarray} and
\begin{equation}\label{eq norm spin}  \| a \| ^2 := \langle a,a\rangle
+ (\langle a,a\rangle^2 - |\langle a,\overline{a}\rangle|^2)^{\frac 12}.
\end{equation}
for all $a,b,c\in V$. Let $\|\cdot\|_2$ denote the Hilbert norm of $V$. Clearly,
\begin{eqnarray}\label{23}
\|a\|_2\leq \|a\|\leq \sqrt{2}\|a\|_2,\quad\forall\,a\in V.
\end{eqnarray}
We note that the spin factor $V$ is not strictly convex.\smallskip

There is an undoubted advantage of regarding spin factors as projective tensor products of certain Hilbert spaces. Given two Banach spaces $X$ and $Y$, the symbol $X\otimes_{\pi} Y$ will denote the projective tensor product of $X$ and $Y$, while $\|.\|_{\pi}$ will stand for the projective norm. It is known that $$\|u\|_{\pi} = \inf\left\{\sum_{i=1}^{m} \|x_i\| \|y_i\| : \sum_{i=1}^{m} x_i\otimes y_i  \right\},$$ for every $u\in X\otimes Y$, and
$\left(X \otimes_\pi Y\right)^{*} = L(X,Y^*)$ (see \cite[\S 2.2]{Ryan}). It follows from this identification that, given a Hilbert space ${H}$, the projective tensor product ${H}\otimes_\pi {H}$ satisfies $({H}\otimes_\pi {H})^*= B({H})$, the space of bounded linear operators on ${H}$.
It follows from the uniqueness of the predual of every von Neumann algebra that ${H}\otimes_\pi {H}=L^1({H})$, the trace class operators on $H$.\smallskip

Spin factors can be represented as real projective tensor products of certain Hilbert spaces. More concretely, by Lemma 3.5 in \cite{BP05}, for every spin factor $V$ there exists a real Hilbert space $K$ such that $V$ is JB$^*$-triple isometrically isomorphic to $\mathbb{C} \otimes_{\pi}^{\mathbb{R}} K$, the real projective tensor product of $K$ and $\mathbb{C}$, when the latter is regarded as a real space. Under this point of view, the UKK property in spin factors can be easily handle.

\begin{theorem}\label{spin01}
Every spin factor satisfies the Uniform Kadec-Klee property.
\end{theorem}

\begin{proof} Let $V = \mathbb{C} \otimes_{\pi}^{\mathbb{R}} K$ be a spin factor, where $K$ is a real Hilbert space. It is known that $\left(\mathbb{C} \otimes_{\pi}^{\mathbb{R}} K\right)^* \cong B_{_\mathbb{R}} (\mathbb{C}_{_\mathbb{R}}, K)$, where the latter denotes the space of all bounded real linear operators from $\mathbb{C}_{_\mathbb{R}}$ into $K$. Since $V$ is reflexive, we can regard $V$ as a dual Banach space with $V_* = V^* = B_{_\mathbb{R}} (\mathbb{C}_{_\mathbb{R}}, K)$.\smallskip

For each finite dimensional subspace $F\subseteq K$, let $p_{_F}$ denote the orthogonal projection of $K$ onto $F$. We define a finite range operator $$K_{_F}: B_{_\mathbb{R}} (\mathbb{C}_{_\mathbb{R}}, K)\to B_{_\mathbb{R}} (\mathbb{C}_{_\mathbb{R}}, K),$$ given by $K_{_F} (T) = p_{_F} T$. Clearly $\|K_{_F}\|\leq 1$ for every $F$ as above. Let $\mathcal{F} (K)$ denote the set of all finite dimensional subspaces of $K$ ordered by inclusion. If we consider the net $(K_{_F})_{_{F\in \mathcal{F} (K)}}$, it can be easily checked that, for each $T\in B_{_\mathbb{R}} (\mathbb{C}_{_\mathbb{R}}, K),$ the net $(\| K_{_F} (T) -T\|)_{F}$ tends to zero.\smallskip

Fix now $S,T\in B_{_\mathbb{R}} (\mathbb{C}_{_\mathbb{R}}, K)$, $r,s\in (0,1]$, and $\lambda\in \mathbb{C}$ with $|\lambda |\leq 1$. The inequality $$ \| r K_{_F} (S) (\lambda) + s (T-K_{_F} (T)) (\lambda) \|^2 = \| r p_{_F} S (\lambda) + s (Id-p_{_F}) T (\lambda) \|^2$$  $$ = r^2  \| p_{_F} S (\lambda) \|^2 + s^2 \| (Id-p_{_F}) T (\lambda) \|^2 \leq r^2 \|T\|^2 + s^2 \|S\|^2,$$ holds for every $F\in \mathcal{F} (K)$. Therefore we prove that $$\limsup_{F\in \mathcal{F} (K)} \sup_{\|S\|,\|T\|\leq 1} \left\| r K_{_F} (S) + s (T-K_{_F} (T)) \right\|\leq \sqrt{r^2+s^2}\leq 1,$$ for every $r,s\in (0,1]$ with $r^2+s^2\leq 1$.\smallskip

Pick now  $1\otimes x + i \otimes y$ in $\mathbb{C} \otimes_{\pi}^{\mathbb{R}} K$ and $T\in B_{_\mathbb{R}} (\mathbb{C}_{_\mathbb{R}}, K)$. It can be easily seen that $$ K_{_F}^* (1\otimes x + i \otimes y) (T) = (1\otimes x + i \otimes y) K_{_F} (T) = (1\otimes x + i \otimes y)  (p_{_F} T) $$ $$= \langle x, p_{_F} T(1)\rangle +  \langle y, p_{_F} T(i)\rangle =
 \langle p_{_F}(x) ,  T(1)\rangle + \langle p_{_F}(y) ,  T(i)\rangle $$ $$= (1\otimes p_{_F}(x) + i \otimes p_{_F}(y)) (T),$$ and hence $K_{_F}^* (1\otimes x + i \otimes y) = 1\otimes p_{_F}(x) + i \otimes p_{_F}(y),$ for every $1\otimes x + i \otimes y\in \mathbb{C} \otimes_{\pi}^{\mathbb{R}} K.$ Consequently, $$ \left\| K_{_F}^* (1\otimes x + i \otimes y) - (1\otimes x + i \otimes y) \right\|_{\pi} \leq \|p_{_F} (x)-x\|_{K} \ \|p_{_F} (y)-y\|_{K},$$ and hence $\displaystyle \lim_{F\in \mathcal{F} (K)} \left\| K_{_F}^* (1\otimes x + i \otimes y) - (1\otimes x + i \otimes y) \right\|_{\pi}=0$.\smallskip

Applying Theorem \ref{proposition 3.1.} with $X= B_{_\mathbb{R}} (\mathbb{C}_{_\mathbb{R}}, K)$ and the net $(K_{F})_{F\in \mathcal{F} (K)}$, we deduce that $X$ satisfies the $(r,s))$-LKK$^*$ property for every $r,s\in (0,1]$ with $r^2 +s^2\leq 1$. Proposition \ref{Proposition 2.4} proves that $X^* = V$ satisfies the UKK$^*$ property. The proof concludes by observing that for a reflexive space the UKK property and the UKK$^*$ property are equivalent.
\end{proof}

\begin{corollary}\label{c UKK equals KKP}
The Kadec-Klee property and the uniform Kadec-Klee property are equivalent for JB$^*$-triples.
\end{corollary}

\begin{proof} Let $E$ be a JB$^*$-triple.  Suppose that $E$  has the Kadec-Klee property, by \cite[Proposition 2.13]{BP04}, we can derive that $E$ is finite-dimensional or a spin factor or a Hilbert space. It is known that a Hilbert space satisfies the uniform Kadec-Klee property. If $E$ is a finite dimensional JB$^*$-triple,
then it is nearly uniform convex, and hence $E$ satisfies the uniform Kadec-Klee property (see, for example, \cite[p.744]{Hu}). Finally, if $E$ is a spin factor, Theorem \ref{spin01} proves that $E$ has the uniform Kadec-Klee property.
\end{proof}

We can actually show that similar techniques to those employed above can be also applied to give and alternative proof of a result due to C. Lennard \cite{Le1}. As in the proof given by Lennard in the just quoted paper, we rely on previous results of J. Arazy \cite{Ar}. We have already commented that the projective tensor product, $L^1 (H) =H\otimes_{\pi} H,$ of a Hilbert space $H$ with itself coincides with the predual of $B(H)$. In this setting $L^1 (H) =H\otimes_{\pi} H$ also is the dual of the space $K(H)$ of compact operators on $H$.

\begin{theorem}\label{t Lennard trace class}\cite[Theorem 2.4]{Le1}
For every Hilbert space $H$, the space $L^1 (H)= H \otimes_{\pi} H$ of all trace class operators on $H$ satisfies the UKK$^*$ property.
\end{theorem}

\begin{proof} Keeping in mind the notation in the proof of Theorem \ref{spin01}, For each $F\in \mathcal{F} (H)$, let $K_{_F}: K (H)\to K(H),$ denote the finite rank (contractive) projection given by $K_{_F} (T) = p_{_F} T p_{_F}$, where $p_{F}$ denotes the orthogonal projection of $H$ onto $F$, and by an abuse of notation we regard $T p_{_F}$ as mapping from $F$ into $H$. The net $(K_{_F})_{_{F\in \mathcal{F} (H)}}$ satisfies that, for each $T\in K(H)$, the net $(\| K_{_F} (T) -T\|)_{F}$ converges to zero.\smallskip

Let us take an element $u\in L^1 (H) =H\otimes_{\pi} H,$ with $\|u\|_{\pi}\leq 1$, $r,s\in (0,1]$, and $F\in \mathcal{F} (H)$. By \cite[page 48 $(iii)$]{Ar} (or \cite[Proposition 2.2]{Le1} or \cite[Lemma 3.1]{BP05}) it follows that $$r \left\| r K_{_F}^* (u)\right\|_{\pi}  + s \left\| (u-K_{_F}^* (u))  \right\|_{\pi} = r \left\| p_{_F} u p_{_F} \right\|_{\pi}+ s \left\|(u-p_{_F} u p_{_F}) \right\|_{\pi}  $$ $$\leq r \left\| p_{_F} u p_{_F} \right\|_{\pi}+ s \left(\left\|p_{_F} u (1-p_{_F}) \right\|_{\pi} + \left\|(1-p_{_F}) u p_{_F} \right\|_{\pi} + \left\|(1-p_{_F}) u (1-p_{_F}) \right\|_{\pi}\right)$$ $$\leq r + s \sqrt{3} \left(\|u\|_{\pi}^2 -\left\| p_{_F} u p_{_F} \right\|_{\pi}^2 \right)^{\frac12} \leq r +\sqrt{6} s.$$

On the other hand, it can be easily checked that for each $\xi,\eta\in H$, we have $K_{_F}^* (\xi\otimes \eta) = p_{_F}(\xi) \otimes p_{_F}(\eta)$, and hence $$ \left\| K_{_F}^* (\xi \otimes \eta) - (\xi \otimes \eta) \right\|_{\pi} \leq \|p_{_F} (\xi)-\xi\|_{H} \ \|p_{_F} (\eta)-\eta\|_{H},$$ which can be apply to prove that $\displaystyle \lim_{F\in \mathcal{F} (H)} \left\| K_{_F}^* (u) - (u) \right\|_{\pi} =0$, for each $u$ in $H\otimes_{\pi} H$.\smallskip

We deduce by Theorem \ref{proposition 3.1.} (see also Remark \ref{remark shrinking compact approximation inequality in the dual space}) that $K(H)$ satisfies the $(r,s)$-LKK$^*$-property for every $r,s\in (0,1]$ with $r +\sqrt{6} s\leq 1$. Since the unit ball of $L^1 (H) =H\otimes_{\pi} H$ is weak$^*$ sequentially compact (see \cite[Lemma 2.3]{Le1}), we are in position to apply Proposition \ref{Proposition 2.4} to assure that $L^1 (H) =H\otimes_{\pi} H$ satisfies the UKK$^*$.
\end{proof}

Let us briefly recall that a \emph{real JB$^*$-triple} is a closed real subspace $A$ of a JB$^*$-triple $E$ which is closed for the triple product of $E$ (see \cite{IsKaRo}). There is an equivalent definition asserting that every real JB$^*$-triple is precisely the real subspace of all fixed points of a conjugation on a complex JB$^*$-triple. More concretely, a conjugation on a complex Banach space $X$ is a conjugate linear isometry of order 2 $\tau : X\to X$. The associated \emph{real form} of $X$ is the set of all $\tau$-fixed points in $X$
$X^{\tau} := \{ x\in X : \tau (x) =x\}.$  Let us observe that $X^{\tau}$ is the image of the real contractive projection $\frac{1}{2} (id +
\tau)$, and that \begin{equation}\label{i} X = X^{\tau} \oplus i
X^{\tau}.
\end{equation} The Kaup-Banach-Stone theorem asserts that every linear (or conjugate linear) surjective isometry on a complex JB$^*$-triple $E$ is a triple isomorphism (see \cite[Proposition 5.5]{Ka83}). In particular, for every conjugation $\tau$ on $E$, $E^{\tau}$ is a real JB$^*$-subtriple of $E$. It is established in \cite[Proposition 2.2]{IsKaRo} that every real JB$^*$-triple is of the form $E^{\tau}$, where $E$ is a complex JB$^*$-triple and $\tau$ is a conjugation on $E$.\smallskip

It is clear that if $X$ is a complex Banach space satisfying the KKP or the UKK, then $X^{\tau}$ satisfies the same property for every conjugation $\tau$ on $X$. The example given in \cite[Example 1.3]{BunPe05} provides a complex Banach space $X$ and a conjugation $\tau$ on $X$ such that $X^{\tau}$ is an infinite dimensional real Hilbert space, and hence satisfies the KKP and the UKK property, while $X$ does not satisfy the KKP.\smallskip

Having in mind the obstacle raised by the previous example, the study of real JB$^*$-triples satisfying the UKK requires of an appropriate strategy and does not follow from Theorem \ref{spin01} nor Corollary \ref{c UKK equals KKP}.

\begin{proposition}\label{p UKK equals KKP for real}
The Kadec-Klee property and the uniform Kadec-Klee property are equivalent for real JB$^*$-triples.
\end{proposition}

\begin{proof} Let $A$ be a real JB$^*$-triple satisfying the KKP. Proposition 3.13 in \cite{BP05} assures that $A$ is finite dimensional or a real, complex, or quaternionic Hilbert space or a real or complex spin factor. Applying the previous results, we can assume that $A$ is a real spin factor (i.e. a real form of a complex spin factor). By \cite[Theorem 4.1$(viii)$]{Ka97}, there exits a real Hilbert space $H$, and closed linear subspaces $H_{1}$ and $H_{2}$ of $H$ such that $H_2={H_1}^{\perp}$, and $A$ is the $\ell_1$-sum $A= H_1 \oplus^{1} H_{2}$.\smallskip

Since $A_* = H_1 \oplus^{\ell_{\infty}} H_{2}$. We consider the set $\mathcal{F} (H_1)\times \mathcal{F} (H_2)$ with the order given by inclusion (i.e. $(F_1,F_2)\leq (G_1,G_2)$ if and only if $F_i\subseteq G_i$). For each $(F_1,F_2)\in \mathcal{F} (H_1)\times \mathcal{F} (H_2)$ we define a mapping $K_{_{(F_1,F_2)}} : H_1 \oplus^{\ell_{\infty}} H_{2} \to H_1 \oplus^{\ell_{\infty}} H_{2}$, defined by $K_{_{(F_1,F_2)}} (x_1,x_2) = (p_{_{F_1}} (x_1), p_{_{F_2}} (x_2)) $, where $p_{_{F_i}}$ is the orthogonal projection of $H_i$ onto $F_i$. Clearly, $K_{_{(F_1,F_2)}}$ is a finite rank projection on $H_1 \oplus^{\ell_{\infty}} H_{2}$ with $\| K_{_{(F_1,F_2)}}\| \leq 1$, and for each $(x_1,x_2)\in H_1 \oplus^{\ell_{\infty}} H_{2}$ we have $\displaystyle \lim_{_{(F_1,F_2)}} \| K_{_{(F_1,F_2)}} (x_1,x_2) - (x_1,x_2)  \| =0$. It is also clear that for each $(x_1,x_2)\in H_1 \oplus^{1} H_{2}$  we have $\displaystyle \lim_{_{(F_1,F_2)}} \| K_{_{(F_1,F_2)}}^* (x_1,x_2) - (x_1,x_2)  \| =0$.\smallskip

It is not hard to see that for each $(x_1,x_2),(y_1,y_2)$ in the closed unit ball of $A_*$, and $r,s\in (0,1]$, we have $$\left\| r  K_{_{(F_1,F_2)}} (x_1,x_2) + s ( (y_1,y_2) - K_{_{(F_1,F_2)}} (y_1,y_2))  \right\|$$ $$ = \max\! \left\{\! \| r p_{_{F_1}} (x_1) + s (y_1- p_{_{F_1}} (y_1))\|_{_{H_1}}, \| r p_{_{F_2}} (x_2) + s (y_2- p_{_{F_2}} (y_2))\|_{_{H_2}} \right\} \leq \sqrt{r^2+s^2}.$$

Finally, combining Theorem \ref{proposition 3.1.} and Proposition \ref{Proposition 2.4} above we obtain that $A$ satisfies the UKK property.
\end{proof}

We have already commented that, F. Rambla and J. Becerra proved in \cite{BeRam09} that when $X$ is a real or complex JB$^*$-triple or the predual of a real or complex JBW$^*$-triple, then following are equivalent:\begin{enumerate}[$(a)$]\item $X$ satisfies the fixed point property;
\item $X$ has normal structure;
\item $X$ is reflexive.
\end{enumerate}
It is further known that when $X$ is a real or complex JB$^*$-triple, then $X$ has normal structure if and only if $X$ has weak normal structure. It was known that the class of real or complex JB$^*$-triples satisfying the KK property is strictly smaller than the class of reflexive real or complex JB$^*$-triples (compare \cite[Proposition 2.13]{BP04}). We can conclude now that the class of JB$^*$-triples satisfying the UKK property is precisely the class determined in \cite[Proposition 2.13]{BP04} and \cite[Proposition 3.13]{BP05}.

\end{document}